\documentclass[12pt,a4paper]{amsart}
\usepackage{enumerate}
\usepackage{graphicx}
\usepackage{yhmath}
\usepackage{amsmath,verbatim}
\usepackage{amssymb}
\usepackage{pgf,tikz}
\usetikzlibrary{arrows}
\usepackage[utf8]{inputenc}

\theoremstyle{plain}
\newtheorem{thm}{Theorem}[section]
\newtheorem{lemma}[thm]{Lemma}

\newtheorem{prop}[thm]{Proposition}
\newtoks\prt

\theoremstyle{definition}

\newtheorem{remark}[thm]{Remark}

\newtheorem{ques}[thm]{Question}
\newtheorem{definition}[thm]{Definition}

\def\eqn#1$$#2$${\begin{equation}\label#1#2\end{equation}}

\numberwithin{equation}{section}

\def\dist{\operatorname{dist}}

\def\co{\operatorname{co}}
\def\spn{\operatorname{span}}

\def\T{\mathcal{T}}

\def\H{\mathcal{H}}
\def\phi{\varphi}
\def\epsilon{\varepsilon}
\def\en{\mathbb N}
\def\er{\mathbb R}

\def\id{\operatorname{id}}
\def\sgn{\operatorname{sgn}}
\def\rn{\mathbb R^n}
\def\sgn{\operatorname{sgn}}

\newtoks\by
\newtoks\paper
\newtoks\book
\newtoks\jour
\newtoks\yr
\newtoks\pages
\newtoks\vol
\newtoks\publ
\def\ota{{\hbox\vol{???}}}
\def\cLear{\by=\ota\paper=\ota\book=\ota\jour=\ota\yr=\ota
	\pages=\ota\vol=\ota\publ=\ota}
\def\endpaper{\the\by, {\the\paper},
	\textit{\the\jour} \textbf{\the\vol} (\the\yr), \the\pages.\cLear}
\def\endbook{\the\by, \textit{\the\book}, \the\publ.\cLear}
\def\endprep{\the\by, \textit{\the\paper}, \the\jour.\cLear}
\def\endyearprep{\the\by, \textit{\the\paper}, \the\jour, (\the\yr).\cLear}
\def\name#1#2{#2 #1}
\def\nom{ \rm no. }

\MakeRobust{\ref}

\makeatletter
\newcommand{\labeltext}[2]{%
	\@bsphack
	\def\@currentlabel{#1}{\label{#2}}%
	\@esphack
}
\makeatother

\def\step#1#2#3{\par \noindent{{\\ \bf Step~\labeltext{#1}{#3}#1. }{\bf #2. }}}

\begin{document}
	
	\title[Smooth homeomorphic approximation of piecewise affine homeomorphisms]{Smooth homeomorphic approximation of piecewise affine homeomorphisms}
	
	\author[D. Campbell]{Daniel Campbell}
	\address{D.~Campbell: Department of Mathematics, University of Hradec Kr\' alov\' e, Rokitansk\'eho 62, 500 03 Hradec Kr\'alov\'e, Czech Republic} 
	\address{Faculty of Economics, University of South Bohemia, Studentsk\' a 13, Cesk\' e Budejovice, Czech Republic}
	\email{\tt daniel.campbell@uhk.cz}
	
	\author[F. Soudsk\' y]{Filip Soudsk\' y}
	\address{F.~Soudsk\' y:Department of mathematics and didactic of mathematics, Technical University of Liberec
		Faculty of Science, Humanities and Education}
	\address{Department of Mathematics, University of Hradec Kr\' alov\' e, Rokitansk\'eho 62, 500 03 Hradec Kr\'alov\'e, Czech Republic} 
	\email{\tt filip.soudsky@tul.cz}

	\thanks{The authors were both supported by the grant GACR 20-19018Y}

	
	\begin{abstract}
		Given any $f$ a locally finitely piecewise affine homeomorphism of $\Omega \subset \rn$ onto $\Delta \subset \rn$ in $W^{1,p}$, $1\leq p < \infty$ and any $\epsilon >0$ we construct a smooth injective map $\tilde{f}$ such that $\|f-\tilde{f}\|_{W^{1,p}(\Omega,\rn)} < \epsilon$.
	\end{abstract}
	
	\maketitle
	
	\section{Introduction}
	
	It is well known that any Sobolev map $W^{1,p}$ on $\rn$ can be approximated (for example by mollification) by smooth maps. Standard procedures for obtaining a smooth approximation of a Sobolev homeomorphism, however, do not preserve the injectivity of the original map. Therefore if we want to approximate a Sobolev homeomorphism by a smooth and injective map we need to introduce more creative approaches. This question is known as the Ball-Evan's problem (see \cite{B}, \cite{B2}) and may be formulated as follows:
	\begin{ques}
		If $f\in W^{1,p}(\Omega, \mathbb R^n)$ is invertible, can it be approximated in $W^{1,p}(\Omega,\er^n)$ by piecewise affine invertible mappings (or diffeomorphisms)?
	\end{ques}

	The motivation for this question is found in non-linear elasticity, especially concerning the question of regularity of minimizers, numerical approximation of minimizers and the Euler–Lagrange equations. For a review of these questions we recommend \cite{B2}. In short, many of the approaches we would like to use, require testing with a smooth map close to our candidate. On the other hand, this map should still be an admissible map for the nonlinear elastic problem. This leads us to the problem of finding diffeomorphisms close to homeomorphisms in the Sobolev norm.
	
	The question quickly arises whether approximation by piecewise affine homeomorphisms or by diffeomorphisms is equivalent. One implication was shown in \cite{IO2}, namely that the triangulation of a diffeomorphism is possible and their result covers all dimensions $n\geq 1$. They prove that their piecewise affine maps are close to the original diffeomorphism in $W^{1,\infty}$ and the same holds for the inverse. The opposite implication is more difficult and so far only a planar result \cite{MP} has been proven. Note that recently a planar result working also with the second derivatives (of the map but not its inverse) has been achieved in \cite{CH}.

	In this paper we show that a piecewise affine map can be approximated by smooth homeomorphic maps in $W^{1,p}$ for any $1\leq p< \infty$. Our proof is simpler and more concise than the proof in \cite{MP}. Further our result holds for all $n\geq 2$. The drawback of our approach concerns the inverse of the map. Although our maps are smooth their inverses are not. Despite this, given a fixed $1\leq q<\infty$ it is not hard to adapt our approach so that our smooth approximations also approximate the inverse of the piecewise affine maps in $W^{1,q}$, but we have not pursued this.
	
	Positive results on diffeomorphic approximation of homeomorphisms in $W^{1,p}$ are up till now exclusively planar (but see also \cite{IO}) and a survey can be found in \cite{CHT}. For dimension 3 the question is completely open and in dimension higher $n\geq 4$ there is only known the negative result \cite{CHT}. In this respect our result is a first building block towards attacking the very challenging question of smooth injective approximation in higher dimension.
	
	The current results which allow one to approximate a map and its inverse simultaneously in $W^{1,p}$ are limited in their scope. In \cite{DP} one must assume that the map is bi-Lipschitz. Without this assumption the result is only known for the $W^{1,1}-W^{1,1}$ case, proved in \cite{BiP}. Although being able to approximate a piecewise affine map and its inverse simultaneously by diffeomorphisms is useful, we would need be able to approximate homeomorphisms and their inverse simultaneously to utilise this result. On the other hand the path to approximating a homeomorphism $f \in W^{1,p}$ (say from $\er^3$ to $\er^3$) by smooth homeomorphisms may well lead via piecewise affine maps. Our result opens this pathway; it is as follows:

	\begin{thm}\label{Cor}
		Let $\Omega\subset\er^n$ be a domain. Let $f:\Omega\to\er^n$ be a locally finite piece-wise affine homeomorphism. Then for every $\epsilon>0$ there exists an injective $\tilde{f} \in \mathcal{C}^1(\Omega, \rn)$ with 
		$$\|\tilde{f} - f\|_{L^{\infty}(\Omega)} < \epsilon$$ and
		$$
		\|D\tilde{f} - Df\|_{L^p(\Omega)}< \epsilon.
		$$
	\end{thm}
	
	In fact this theorem is the corollary of a the more general case in the setting of rearrangement invariant function spaces. Given a rearrangement invariant function space $X$ we define the characteristic function $\phi_X$ of $X$ as $\phi_X(t) = \|\chi_A\|_{X}$ for any $A\subset \rn$ satisfying $\mathcal{L}^n(A) = t$. In the special case of $W^{1,p}$ we have $\phi_{L^{p}}(t) = \sqrt[p]{t}$. We prove:
	
	\begin{thm}\label{main}
		Let $\Omega\subset\er^n$ be a domain. Let $f:\Omega\to\er^n$ be a locally finite piece-wise affine homeomorphism. Let $X(\Omega)$ be a rearrangement invariant function space with $\phi_X(t) \to 0$ as $t\to 0$. Then for every $\epsilon>0$ there exists an injective $\tilde{f} \in \mathcal{C}^1(\Omega, \rn)$ with 
		$$\|\tilde{f} - f\|_{L^{\infty}(\Omega)} < \epsilon$$ and
		$$
		\|D\tilde{f} - Df\|_{X(\Omega)}< \epsilon.
		$$
	\end{thm}
	
	We conclude this introduction with an overview of the idea of the proof in the simpler case of $n=2$. This case suffices to demonstrate the idea we apply in higher dimension and the simplicity makes it very easy to follow. Later, in section~\ref{Meat} we give a full detailed proof of the claim in dimension 3. Finally we explain the differences for general dimension.
	
	\subsection*{Idea of the proof for $n=2$}
	
	\step{1}{The proof of the claim given a map $\Xi_{\delta}$}{1A}
	
	Let us assume that we have a locally finite simplicial complex of triangles (for the definition see the preliminaries) and a homeomorphism $f$ which equals an affine map on each triangle. By dividing the triangles we may assume that there are no angles greater than $\frac{\pi}{2}$.
	
	We choose a parameter $\delta>0$. Our aim is to construct a smooth injective mapping, (call it $\Xi_{\delta}$) which sends each side of each triangle of the simplicial complex onto itself. At the same time we want
	\begin{equation}\label{outer}
	\partial_{\vec{n}}\Xi_{\delta}(t) = 0
	\end{equation}
	for any $t\in \partial T$, where $T$ is a triangle of the simplicial complex and $\vec{n}$ is a unit vector perpendicular to $\partial T$ at $t$. At the vertices of $T$ we interpret this as $D\Xi_{\delta} = 0$. Further, calling $E_{\delta} = \{\Xi_{\delta} \neq \id \}$, we require that
	\begin{equation}\label{ids}
	\mathcal{L}^2(E_{\delta}\cap T) \leq C\delta \mathcal{L}^2(T) 
	\end{equation}
	for any triangle $T$. Finally we require that
	\begin{equation}\label{bound}
	\|D\Xi_{\delta}\|_{\infty} \leq C
	\end{equation}
	with $C$ independent of $\delta$.

	Given a $\Xi_{\delta}$ as described in the previous paragraph we define
	$$
	\tilde{f}(x) = f(\Xi_{\delta}(x)).
	$$
	Although $f$ is not smooth it has only a jump at $\partial T$ and that only for the derivatives not perpendicular to $\vec{n}$. But we have \eqref{outer}, which (using the chain rule) implies that $\tilde{f}$ is smooth everywhere (including on every $\partial T$). Combining \eqref{ids} and \eqref{bound} we have 
	$$
	\int_{T\cap E_{\delta}} |Df - D\tilde{f}|^{p} \leq C \int_{T\cap E_{\delta}}|Df|^{p}\leq  C\delta \int_{T}|Df|^{p}
	$$
	and so
	$$
	\|Df - D\tilde{f}\|_{p} \leq C\delta^{1/p} \|Df\|_{p} < \epsilon
	$$
	for $\delta$ sufficiently small. What remains is to show the construction of such a $\Xi_{\delta}$.

	\step{2}{Determining some constants}{1B}
	
	For any $s$ vertex of any triangle $T$ there are a finite number of triangles containing $s$. Then there is a minimal angle between any two sides ending at $s$, call it $\alpha_1(s)$ (if this angle is greater than $\pi/4$ then put $\alpha_1(s) = \pi /4$). For each vertex $s$ define
	$$
	\alpha_2(s):=\min\{\alpha_1(s),\alpha_1(s'):s'\textup{ neighbour of }s\}.
	$$
	Then for any side $S$ with endpoints $s_1, s_2$ we call
	$$
	\alpha(S) = \min\{ \alpha_2(s_1), \alpha_2(s_2) \}.
	$$
	Further, for $S$, the common side of the triangles $T_1, T_2$ we define 
	$$
	m_1(S) = \min \Big\{ 1, \H^1(S), \frac{\mathcal{L}^2(T_1)}{\H^1(S)}, \frac{\mathcal{L}^2(T_2)}{\H^1(S)}  \Big\}.
	$$
	Further, for any $s$ vertex, we define $m_2(s)$ as the minimum of $m_1(S)$ over all $S$ ending at $s$ and finally for each $S$ with endpoint $s_1, s_2$ we define
	$$
	m(S) = \min \{m_2(s_1), m_2(s_2)\}.
	$$
	To make our notation shorter we define
	$$
	d = d(S) = \tfrac{1}{2}\delta m(S)
	$$
	and
	$$
	t = t(S) = \tan(\tfrac{1}{3}\alpha(S)).
	$$
	For each vertex $s$ we define $\eta_1(s) = \min\{\sqrt{ \mathcal{L}^2(T)};T \ni s \}$ and $\eta_2(s)$ as the distance from $s$ to the next closest vertex. Then we define
	$$
	\eta(s) = \min\{\eta_1(s), \tfrac{1}{2}\eta_2(s)\}.
	$$
	
	\step{3}{A map that pinches around the middle bit of each side $S$}{1C}
	
	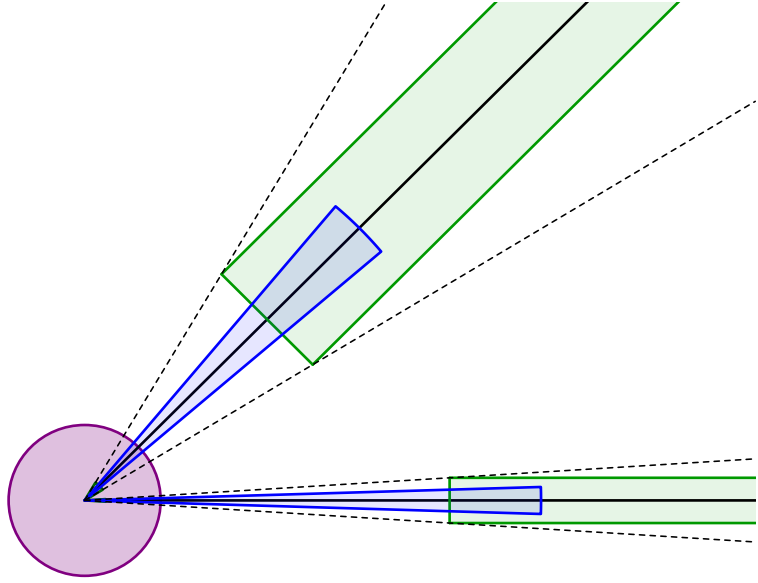
\begin{figure}
		\definecolor{qqwuqq}{rgb}{0.,0.39215686274509803,0.}
		\definecolor{qqqqff}{rgb}{0.,0.,1.}
		\definecolor{qqzzqq}{rgb}{0.,0.6,0.}
		\definecolor{yqqqyq}{rgb}{0.5019607843137255,0.,0.5019607843137255}
		\begin{tikzpicture}[line cap=round,line join=round,>=triangle 45,x=0.6cm,y=0.6cm]
		\clip(-2,-2.) rectangle (14.7,11.);
		\draw [line width=1.pt,color=yqqqyq,fill=yqqqyq,fill opacity=0.25] (0.,0.) circle (1.cm);
		\fill[line width=1.pt,color=qqzzqq,fill=qqzzqq,fill opacity=0.10000000149011612] (8.,0.5) -- (8.,-0.5) -- (15.,-0.5) -- (15.,0.5) -- cycle;
		\fill[line width=1.pt,color=qqzzqq,fill=qqzzqq,fill opacity=0.10000000149011612] (3.,5.) -- (5.,3.) -- (16.,14.) -- (14.,16.) -- cycle;
		\draw [shift={(0.,0.)},line width=1.pt,color=qqwuqq,fill=qqwuqq,fill opacity=0.10000000149011612] (0,0) -- (30.96375653207352:0.4582901385157742) arc (30.96375653207352:59.03624346792648:0.4582901385157742) -- cycle;
		\draw [shift={(0.,0.)},line width=1.pt,color=qqwuqq,fill=qqwuqq,fill opacity=0.10000000149011612] (0,0) -- (-3.5763343749973515:0.4582901385157742) arc (-3.5763343749973515:3.5763343749973515:0.4582901385157742) -- cycle;
		\draw [line width=1.pt] (0.,0.)-- (15.,0.);
		\draw [line width=1.pt] (0.,0.)-- (15.,15.);
		\draw [line width=1.pt,color=qqzzqq] (8.,0.5)-- (8.,-0.5);
		\draw [line width=1.pt,color=qqzzqq] (8.,-0.5)-- (15.,-0.5);
		\draw [line width=1.pt,color=qqzzqq] (15.,-0.5)-- (15.,0.5);
		\draw [line width=1.pt,color=qqzzqq] (15.,0.5)-- (8.,0.5);
		\draw [line width=1.pt,color=qqzzqq] (3.,5.)-- (5.,3.);
		\draw [line width=1.pt,color=qqzzqq] (5.,3.)-- (16.,14.);
		\draw [line width=1.pt,color=qqzzqq] (16.,14.)-- (14.,16.);
		\draw [line width=1.pt,color=qqzzqq] (14.,16.)-- (3.,5.);
		\draw [shift={(0.,0.)},line width=1.pt,color=qqqqff,fill=qqqqff,fill opacity=0.10000000149011612]  (0,0) --  plot[domain=0.702256931509007:0.8685393952858895,variable=\t]({1.*8.514693182963201*cos(\t r)+0.*8.514693182963201*sin(\t r)},{0.*8.514693182963201*cos(\t r)+1.*8.514693182963201*sin(\t r)}) -- cycle ;
		\draw [shift={(0.,0.)},line width=1.pt,color=qqqqff,fill=qqqqff,fill opacity=0.10000000149011612]  (0,0) --  plot[domain=-0.029991004856878334:0.029991004856877897,variable=\t]({1.*10.004498987955369*cos(\t r)+0.*10.004498987955369*sin(\t r)},{0.*10.004498987955369*cos(\t r)+1.*10.004498987955369*sin(\t r)}) -- cycle ;
		\draw [line width=0.6pt,dash pattern=on 2pt off 2pt,domain=0.0:14.7] plot(\x,{(-0.--3.*\x)/5.});
		\draw [line width=0.6pt,dash pattern=on 2pt off 2pt,domain=0.0:14.7] plot(\x,{(-0.--0.5*\x)/8.});
		\draw [line width=0.6pt,dash pattern=on 2pt off 2pt,domain=0.0:14.7] plot(\x,{(-0.-0.5*\x)/8.});
		\draw [line width=0.6pt,dash pattern=on 2pt off 2pt,domain=0.0:14.7] plot(\x,{(-0.--5.*\x)/3.});
		\begin{scriptsize}
		\end{scriptsize}
		\end{tikzpicture}
		\caption{The green rectangles are the sets where $b_{S,\delta}$ are not identity. These sets are inside circular sectors around the respective side $S$ with angular width $\tfrac{1}{3}\alpha(S)$ either side of $S$. The angle between sides is at least $\alpha_1(s) \geq \alpha(S)$ and so the two green rectangles do not intersect. The blue circular sectors are the sets where $a_{s,S,\delta} \neq \id$. Similarly as before the sets are pairwise disjoint. The purple circle is the set where $e_{s,\delta}\neq \id$.}\label{fig:only}
	\end{figure}

	By $S$ we denote the common side of $T_1$ and $T_2$, whose endpoints we denote as $s$ and $s + l\vec{u}$, where $\vec{u}$ is the unit vector parallel to $S$. By $\vec{n}$ we denote a unit vector perpendicular to $S$. We define a mapping $b_{S,\delta}$ using 2 functions. The first function, $g_{dt} = g_{d(S)t(S)}$, is a smooth injective odd function (defined in Proposition~\ref{GString}), equal to identity for arguments greater than $dt = d(S)t(S)$ and $g'_{dt}(0) = 0$. The other is a smooth convex combination $c_{a,b}$ (defined in Proposition~\ref{CString}). The mapping $b_{S,\delta}$ smoothly `pinches' in the direction $\vec{n}$ along the middle part of $S$ and is identity as soon as we are not too close to $S$. For $w \in (-\infty ,l/2]$ and $z\in \er$ we define
	$$
	b_{S,\delta}(s + w\vec{u} + z \vec{n}) = s + w\vec{u} + c_{d,2d}(w) g_{dt}(z)\vec{n}  +[1- c_{d,2d}(w)]z\vec{n}
	$$
	and similarly for $w \in (-\infty ,l/2]$ and $z\in \er$
	$$
	b_{S,\delta}(s + l\vec{u} - w\vec{u} + z \vec{n}) = s +l\vec{u} - w\vec{u} + c_{d,2d}(w) g_{dt}(z)\vec{n}  +[1- c_{d,2d}(w)]z\vec{n}.
	$$
	For each $S$ we then have $b_{S,\delta}$ defined on $\er^2$.
	
	Thanks to the properties of $g_{dt}$ we have that $\partial_{\vec{n}}b_{S,\delta} = 0$ on $S \setminus [B(s,2d) \cup B(s+l\vec{u}, 2d)]$. It is not hard to check (see Figure~\ref{fig:only}) that for any $S\neq S'$ we have that
	\begin{equation}\label{reading1}
	\{b_{S,\delta} \neq \id\}\cap \{b_{S',\delta}\neq \id \} = \emptyset.
	\end{equation}
	Further  $\|g_{dt}'\|_{\infty}\leq C$ with $C$ independent of $dt$ and so $\|Db_{S,\delta}\|_{\infty}\leq C$. Then the smooth injective mapping $b_{\delta}$ defined as the composition of all $b_{S,\delta}$  for all $S$ (the composition is independent of the order we take the composition in by \eqref{reading1}) satisfies
	\begin{equation}\label{bound1}
	\|Db_{\delta}\|_{\infty}\leq C.
	\end{equation}
	Notice that
	\begin{equation}\label{cutie1}
	\mathcal{L}^2(\{ b_{S,\delta} \neq \id\}) \leq 2d(S)t(S) \H^1(S)  \leq \delta \mathcal{L}^2(T)
	\end{equation}
	by the choice of $m(S)$, i.e. $m(S) \leq  \frac{\mathcal{L}^2(T)}{\H^1(S)}$ and $\alpha(S)$ i.e. $t(S) =\tan(\alpha(S)/3)\leq 1$.

	\step{4}{A map that pinches at the end parts of each side $S$}{1D}

	For each $s$ one of the endpoints of a side $S$ we define a map $a_{s,S, \delta}$. The map $a_{s,S,\delta}$ will pinch along the side $S$ in the direction $\vec{n}$ on $S\cap B(s, 2d)$ (i.e. the parts of $S$ not pinched by $b_{S,\delta}$). Again we call $\vec{u}$ the vector parallel to $S$ going from $s$ and $\vec{n}$ will be a unit vector perpendicular to $\vec{u}$. Introducing the polar coordinate system centred at $s$ from the vector $\vec{u}$ i.e. $(r,\phi) = s+  r\cos(\phi)\vec{u} + r\sin(\phi)\vec{n}$ we define
	$$
	a_{s,S, \delta}(r,\phi) =  (1-c_{2d,3d}(r))(r, g_{\alpha/3}(\phi)) + c_{2d,3d}(r) (r,\phi).
	$$
	Standard computation gives that
	\begin{itemize}
		\item $a_{s,S,\delta}$ is injective on $\er^2$,
		\item $a_{s,S,\delta}$ is smooth everywhere except at $s$,
		\item $\|Da_{s,S,\delta}\|_{\infty} \leq C$ with $C$ independent of $s, S$ and $\delta$,
		\item $\partial_{\vec{n}}a_{s,S,\delta} = 0$ on $S\cap B(s,2d)$.
	\end{itemize}
	As can be easilly observed from Figure~\ref{fig:only}, we have
	\begin{equation}\label{reading2}
	\{a_{s,S,\delta} \neq \id\}\cap \{a_{s', S',\delta}\neq \id \} = \emptyset.
	\end{equation}
	as soon as either $s \neq s'$ or $S \neq S'$. Further a standard computation gives $\|Da_{s,S,\delta}\| \leq C$ independently of $s, S$ and $\delta$. Then $a_{\delta}$ defined as the composition of $a_{s,S,\delta}$ for all combinations of $s$ and $S$ (the order of the composition is not important by \eqref{reading2}) satisfies
	\begin{equation}\label{bound2}
	\|Da_{\delta}\|_{\infty}\leq C.
	\end{equation}
	Then we have that $a_{\delta}\circ b_{\delta}$ satisfies $\partial_{\vec{n}}a_{\delta}\circ b_{\delta}$ at all points of each $S$ and 
	\begin{equation}\label{smooth}
	a_{\delta}\circ b_{\delta} \text{ is smooth everywhere except at vertices of the simplicial complex.}
	\end{equation}
	Similarly as before we have
	\begin{equation}\label{cutie2}
	\mathcal{L}^2(\{ a_{s,S,\delta} \neq \id\}) \leq2d(S)t(S) \H^1(S)  \leq \delta \mathcal{L}^2(T)
	\end{equation}
	by the choice of $m(S)$.

	\step{5}{A map that pinches around the vertices $s$}{1E}
	
	For each $s$ vertex of the simplicial complex we define $e_{s,\delta}$ as
	$$
	e_{s,\delta}(x) = \begin{cases} s + \frac{x-s}{|x-s|} g_{\delta\eta(s)}(|x-s|) \quad & x\neq s\\
	s & x=s.\\ \end{cases}
	$$
	Then $e_{s,\delta}$ is a radial smooth mapping around $s$ with
	\begin{equation}\label{bound3}
	De_{s,\delta}(s) = 0 \text{ and } \|De_{s,\delta}(s)\|_{\infty}<C
	\end{equation}
	with $C$ independent of $s$ and $\delta$. Further $e_{s,\delta} = \id$ outside $B(s,\eta(s)\delta)$ and so by $2\eta(s) \leq \eta_2(s)$ we have
	$$
	\{e_{s,\delta} \neq \id\}\cap \{e_{s',\delta}\neq \id \} = \emptyset.
	$$
	for any $s\neq s'$. Therefore we define $e_{\delta}$ as the composition of $e_{s,\delta}$ for all $s$ independently on the order of the composition.
	\begin{equation}\label{cutie3}
	\mathcal{L}^2(\{ e_{s,\delta} \neq \id\}) \leq \mathcal{L}^2(B(s,\eta(s)\delta)) \leq \delta^2\eta^2(s) \leq \delta \mathcal{L}^2(T)
	\end{equation}
	by the choice of $\eta(s)$.

	\step{6}{Construction of the map $\Xi_{\delta}$}{1F}
	
	We define $\Xi_{\delta} = e_{\delta} \circ a_{\delta}\circ b_{\delta}$ and show that it satisfies \eqref{outer}, \eqref{ids} and \eqref{bound}. Since each of the maps we compose is injective we have that $\Xi_{\delta}$ is injective. Also each of the maps used in the construction of $\Xi_{\delta}$ sends each side $S$ onto itself and therefore so does $\Xi_{\delta}$. By the Lipschitz regularity of $a_{\delta}\circ b_{\delta}$, \eqref{bound3} and \eqref{smooth} we easilly see that $\Xi_{\delta}$ is smooth everywhere and $\partial_{\vec{n}}\Xi_{\delta} = 0$ everywhere on $\partial T$ for all $T$. Further we have by \eqref{bound1}, \eqref{bound2} and \eqref{bound3} that $\|\Xi_{\delta}\|_{\infty}<C$. For each $T$ we calculate from \eqref{cutie1}, \eqref{cutie2} and \eqref{cutie3} that $$
	\mathcal{L}^2 (\{\Xi_{\delta} \neq \id \} \cap T) \leq C\delta \mathcal{L}^2(T).
	$$
	\qed
	\section{Preliminaries}
	\begin{prop}\label{GString}
		For each $\delta >0$ the function $g_{\delta}: \er \to \er$
		$$
		g_{\delta}(x) =\begin{cases}
		\sgn(x)\big[\tfrac{2}{\delta}|x|^2 - \tfrac{1}{\delta^2}|x|^3\big] & x\in[-\delta,\delta]\\
		x&|x|>\delta.
		\end{cases}
		$$
		Then $g_{\delta}$ is $\mathcal{C}^1$ smooth and increasing on $\er$, $g_{\delta}(0) = 0$, $g_{\delta}'(0) = 0$, $g_{\delta}(\pm\delta) = \pm\delta$ and $g_{\delta}'(\pm\delta) = 1$. Further $g_{\delta}'\leq\tfrac{4}{3}$ on $\er$.
	\end{prop}
	\begin{proof}
		Simple calculation.
	\end{proof}
	
	\begin{prop}\label{CString}
		Define  $c: \er^3 \to \er$ by
		$$
		c_{a,b}(x) =\begin{cases}
		\frac{1+ \sin(\pi\tfrac{x-a}{b-a} - \tfrac{\pi}{2})}{2}& a\leq x\leq b, 0<a<b\\
		1 & x\geq b>a
		\\ 0 & \textup{elsewhere}.
		\end{cases}
		$$
		Then 
		\begin{enumerate}[\upshape(i)]
			\item $c \in\mathcal{C}^1(M)$, where $M=\{(a,b,x):0<a<b, c\in\er\}$.
			\item If $0<a<b$ then $\partial_x c_{a,b}(x)\leq \frac{C}{b-a}$
		\end{enumerate}
	\end{prop}
	\begin{proof}
		Simple calculation.
	\end{proof}

	We use $\co A$ to denote the convex hull of the set $A$, i.e. the smallest convex set containing $A$.
	
	\begin{definition}
		Let $A_1, A_2, \dots A_{n+1} \in  \rn$ be points so that the vectors $A_2 - A_1$, $A_3-A_1, \dots A_{n+1}-A_1$ are linearly independent. Then we call $T = \co \{A_1, A_2, \dots, A_{n+1}\}$, the convex hull of the points $A_1, A_2, \dots, A_{n+1}$ a simplex (or more specifically an $n$-simplex).
		
		Let $\{B_1, B_2, \dots, B_{k+1}\} \subset \{A_1, A_2, \dots A_{n+1}\}$ then we refer to $F =\co \{B_1, B_2, \dots, B_{k+1}\}$ as a $k$-subsimplex of $T$. Given $F$, a $k$-subsimplex of $T$ we denote by $N_F$ the space of vectors perpendicular to the vectors $B_2-B_1$, $B_3-B_1, \dots, B_{k+1}- B_1$.
		
		Let $\Omega\subset \rn$, by the term simplicial complex of $\Omega$ we refer to a countable (possibly finite) set $\T = \{T_1, T_2, \dots\}$ of $n$-simplices with pair-wise disjoint interiors such that $\Omega = \bigcup_{i} T_i$.
		
		We call a simplicial complex locally finite if for any $x\in \Omega$ there exists an $r>0$ such that only a finite number of simplices $T_i$ intersect $B(x,r)$.
		
		We refer to $f$ as a piecewise affine map on $\Omega$ if there exists a $\T$ simplicial complex on $\Omega$ such that $f$ is equal to some affine map $L_i(x) = M_ix + c_i$ on each $T_i \in\T$ and $f$ as a locally finite piecewise affine homeomorphism, if $\T$ is locally finite.
		
		Furthermore for a vector $v\in\er^n$ by $D_v$ ve denote the directional derivative in the direction $v$.
	\end{definition}

	\section{The construction in 3D}\label{Meat}
	
	\begin{thm}\label{SocialConstruct}
		Let $\mathcal{T} = \bigcup_{i=1}^{\infty}T_i$ be a locally finite simplicial complex on $\Omega \subset \er^3$ and let $\delta_i$ be a sequence of positive real numbers. Then there exists a $\Xi_{\delta_i}: \er^3\to \er^3$ such that
		\begin{enumerate}
			\item[$i)$] $\Xi_{\delta_i}\in\mathcal{C}^1(\er^3,\er^3)$,
			\item[$ii)$] $\Xi_{\delta_i}$ is injective,
			\item[$iii)$] $\Xi_{\delta_i}(V) = V$ for every $V$, subsimplex of any $T_i$,
			\item[$iv)$] for every subsimplex $V$ of $T$ we have that $D_{\vec{n}}\Xi_{\delta_i} = 0$ for every $\vec{n} \in N_V$,
			\item[$v)$] there exists an absoloute constant $C$ such that $|D\Xi_{\delta_i}| < C$ on $\er^3$,
			\item[$vi)$] $\Xi_{\delta_i}(x) = x$ on $T_i\setminus (\partial T_i + B(0,\delta_i))$.
		\end{enumerate}
	\end{thm}

	Let $\T$ be a simplicial complex on $\Omega\subset \er^3$. We refer to any $F$, a $2$-subsimplex of $T\in \T$, as a face of the simplex $T$. We call $S$, a $1$-subsimplex of $F$, a $2$-subsimplex of $T$, a side of $F$ resp. a side of $T$. The 0-subsimplices of $\T$ are set of vertices $\{A_1, A_2, \dots\}$.
	
	Let $S$ be a $1$-subsimplex of $\T$ and let $F$ and $F'$ be a pair of $2$-subsimplices of $\T$ such that $S\subset F, F'$. We denote the angle between these two faces at $S$ as $\omega(F,F')$ and call
	$$
		\omega_S = \min\{\tfrac{\pi}{4}, \omega(F,F'); F\cap F' = S, \text{ $F,F'$ both 2-subsimplices of $\T$} \}.
	$$
	
	Let $A$ be a vertex of $T$ and let $S, S'$ be two sides of $T$ both supersimplices of $A$. Then by $\theta(S,S')$ we denote the angle between $S$ and $S'$. We denote
	$$
		\theta_A = \min\{\tfrac{\pi}{4}, \theta(S,S'); S\cap S' = A, \text{ $S,S'$ both 1-subsimplices of $\T$} \}.
	$$
	Further, for each face $F$ containing $A$ there are exactly two sides $A \in S' ,S \subset F$. We denote $\theta_{F,A} = \theta(S,S')$.
	
	Without loss of generality (it suffices to subdivide our original complex) we may assume that $0 < \theta_{F,A} \leq \tfrac{\pi}{2}$ for all pairs $A\subset F$ and $0<\omega(F,F')\le\tfrac{\pi}{2}$ for all pairs $F,F'\subset T \in \T$.
	
	Let $A$ be a 0-subsimplex of $\T$ then we define the number
	$$
	d_A = \min\{|A-\tilde{A}|, \ \tilde{A}\neq A \text{ is a 0-subsimplex of } \T \}.
	$$
	Similarly, let us have a 1-subsimplex $S$ of $\T$ and call $A$ and $A'$ the distinct pair of 0-subsimplices such that $A,A' \subset S$, then we define
	$$
	d_S = \min\{d_{A}, d_{A'}, \dist(S,\T_S)\},
	$$
	where $\T_S$ is the set of all 1-subsimplices of $\T$ which share no endpoints with $S$. We also define $\theta_{S} = \min\{\theta_A, \theta_{A'}, \tfrac{\pi}{4} \}$ where $A$ and $A'$ are the endpoints of $S$. Similarly, let us have a 2-subsimplex $F$ of $\T$ and call $S_1, S_2, S_3$ its sides and $A_1, A_2,A_3$ its endpoints, then we define
	$$
	d_F = \min\{d_{S_1}, d_{S_2}, d_{S_3}, \dist(F,\T_F)\},
	$$
	where $\T_F$ is the set of all 2-subsimplices of $\T$ which share no sides with $F$. Accordingly we define $\omega_F = \min\{\omega_{S_1}, \omega_{S_2}, \omega_{S_3}, \theta_{A_1},\theta_{A_2},\theta_{A_3}\}$.
	
	Let us have $A \subset S \subset F \subset T \in \T$, where $A$ is a vertex of $S$ is a side of $F$ a face of $T$. Then there exists a transformation $\psi_{F,S,A}$ (the composition of a translation with a rotation) so that $\psi_{F,S,A}(A)=(0,0,0)$, $\psi_{F,S,A}(S) \subset\er^+_0\times\{(0,0)\}$, $\psi_{F,S,A}(F)\subset \er^+_0\times \er^+_0\times\{0\}$.
	
	In preparation for the following lemma we define the following sets
	$$
	E_{F,S,A}(\omega, \delta) = \psi_{F,S,A}^{-1}(\{(x_1,r\cos\phi, r\sin \phi); x_1\in [0, \delta], r\in [0,x_1\tan(\tfrac{15\theta_{F,A}}{16})], \phi\in [-\omega,\omega]\}).
	$$

	\begin{lemma}\label{Album0}
		Let $A \in S_1, S'_1 \subset F_1 \subset T_1,T_1' \in \T$, $S'_1\neq S_1$ and $T_1 \neq T_1'$ for $\T$ a locally finite simplicial complex in $\er^3$. Further let $A\subset S_2,S_2' \subset F_2 \subset T_2,T_2' \in \T$, $S_2 \neq S_2'$ and $T_2\neq T_2'$.  Let $\delta>0$ be any number such that $3\delta <d_A$. Then there exists an angle $\alpha_{A}>0$ such that 
		\begin{equation}\label{Contained}
		E_{F_1,S_1,A}(\alpha_A, \delta) \cup E_{F_1,S_1',A}(\alpha_A, \delta) \subset  T_1\cap T_1' \cap B(A, 2\delta)
		\end{equation}
		and 
		\begin{equation}\label{Disjoint}
		\begin{aligned}
		\Big(E_{F_1,S_1,A}\big(\alpha_{A}, \delta\big) \cup E_{F_1,S_1',A}\big(\alpha_{A}, \delta\big)\Big)\cap \Big(E_{F_2,S_2,A}\big(\alpha_{A}, \delta\big)&\cup E_{F_2,S_2',A}\big(\alpha_{A}, \delta\big)\Big)\\
		& \subset  F_1\cap F_2 \cap B(A, 2\delta).
		\end{aligned}
		\end{equation}
		Also
		\begin{equation}\label{ConatinmentE}
		E_{F,S,A}(\alpha_A, \delta) \subset F+B(0,\delta).
		\end{equation}
	\end{lemma}
	\begin{proof}
		
		It holds, for any $0<\delta< \tfrac{d_A}{3}$ that
		\begin{equation}\label{Oasis}
		E_{F_1,S_1,A}(0, \delta) \cup E_{F_1,S_1',A}(0, \delta) \subset  F_1 \cap B(A, 2\delta)
		\end{equation}
		using the definition of $E_{F_1,S_1,A}(\omega, \delta)$, the fact that $\theta_{A} \leq \tfrac{\pi}{4}$ and the fact that $3\delta < d_A$. At each $x$ point of $F$ with distance $r$ from $S_1$ there exists a positive angle $\omega(x)$ such that the entire arc lying in a hyperplane perpendicular to $S_1$ starting at $x$ (and going either above or below the face $F_1$) stays inside a single simplex $T_1$ (or $T_1'$). As $r$ tends to zero we have that $\omega(x) \geq \omega_{S_1}$ by the definition of $\omega_{S_1}$ as the smallest angle between neighbouring faces at $S_1$. The rest of the set $E_{F_1,S_1,A}(0, \delta)$ where we are at least $r_0$ away from $S_1$ is a compact and the function $\omega(x)$ is positive. Therefore it has a positive minimum. The same is true for all sides $S$ and all faces $F$ which intersect our vertex $A$. There are a finite number of such sides and so we get a positive angle $\tilde{\alpha}_A$ such that the arc from $x\in E_{F,S,A}(0, \delta)$ around $S$ through the angle $\tilde{\alpha}_A$ in the hyperplane perpendicular to $S$ stays inside a single $T$ (or $T'$ if we take the arc in the opposite direction), for all $F\supset S \supset A$. The combination of this fact with \eqref{Oasis} gives \eqref{Contained}.

		We now want to prove \eqref{Disjoint}. In order to do this firstly we notice that it holds when we put $\alpha_A = 0$, i.e.
		$$
		\begin{aligned}
		\Big(E_{F_1,S_1,A}\big(0, \delta\big) \cup E_{F_1,S_1',A}\big(0, \delta\big)\Big)\cap \Big(E_{F_2,S_2,A}\big(0, \delta\big)&\cup E_{F_2,S_2',A}\big(0, \delta\big)\Big)\\
		& \subset  F_1\cap F_2 \cap B(A, 2\delta)
		\end{aligned}
		$$
		since each $E_{F,S,A}(0,\delta)\subset F$ and each $E_{F,S,A}(0,\delta)\subset B(A,2\delta)$ by \eqref{Oasis}. Since we want to prove \eqref{Disjoint} for $F_1 \neq F_2$ we will always have that $F_1\cap F_2  = S$ their common side (if such a side exists) or $F_1\cap F_2  = A$.
		
		Let us first consider the case $F_1\cap F_2  = S_1$ a 1-subsimplex of $\T$. By \eqref{Contained} we know that each $E_{F,S,A}$ is contained in 2 neighbouring simplices $T$ and $T'$ sharing the face $F$. From the definition of $\omega_{S_1}$ as (at most) the smallest angle between faces meeting at $S_1$ we know that
		$$
		E_{F_1,S_1,A}(\omega, \delta) \cap E_{F_2,S_1,A}(\omega, \delta) \subset S_1
		$$
		for all $\omega < \tfrac{1}{3}\omega_{S_1}$. The other pairs are dealt with as follows. For $\omega = 0$ the pairs (i.e. $E_{F_1,S_1,A}(0, \delta)$ and $E_{F_2,S_2',A}(0, \delta)$, $E_{F_1,S_1',A}(0, \delta)$ and $E_{F_2,S_1,A}(0, \delta)$ or $E_{F_1,S_1',A}(0, \delta)$ and $E_{F_2,S_2',A}(0, \delta)$) are closed and disjoint and therefore distant. Distance of the sets is a continuous function of $\omega$ and so there is a number $0< {\alpha}_A < \tilde{\alpha}_{A}$) such that all of the above pairs are disjoint.
	\end{proof}

	Call the cylindrical coordinates $\Phi:\er\times[0,\infty)\times \er \to \er^3$ defined by $\Phi(x_1,r,\phi) = (x_1, r\cos\phi, r\sin\phi)$ and by $\Phi^{-1}$ denote the inverse of $\Phi_{|\er\times(0,\infty)\times[-\pi,\pi)}$. In the following we use the following notation to simplify our definitions; for each fixed $A\subset S \subset F \subset T \in \T$ we write
	$$
	\begin{aligned}
	y = y_{F,S,A}(x) &:= [\Phi^{-1}(\psi_{F,S,A}(x))]^1 = [\psi_{F,S,A}(x)]^1 \\
	r = r_{F,S,A}(x) &:= [\Phi^{-1}(\psi_{F,S,A}(x))]^2\\
	\phi = \phi_{F,S,A}(x) &:= [\Phi^{-1}(\psi_{F,S,A}(x))]^3\\
	\end{aligned}
	$$
	For any $A\subset S\subset F$ 0-, 1- and 2-subsimplices of $\T$ we define the map $h_{F,S,A, \delta}$ as follows. Call $t = \tfrac{1}{2}\tan\tfrac{5\theta_{F,A}}{8}$. On the set $E_{F,S,A}(\alpha_{A}, \delta)\cap \{x : 0<y<\delta, 0< r< 3yt\}$ we define
	$$
	h_{F,S,A, \delta}(x) = \psi^{-1}_{F,S,A}\Big(\Phi\big(y, r, (1-c_{\delta/2, \delta}(y))[(1-c_{2yt, 3yt}(r))g_{\alpha_A}(\phi) +  c_{2yt, 3yt}(r)\phi] + c_{\delta/2,\delta}(y)\phi \big) \Big)
	$$
	and elsewhere we define
	$$
	h_{F,S,A, \delta}(x) = x.
	$$
	
	\begin{prop}\label{HorrorMovieI}
		The mapping $h_{F,S,A, \delta}$ has the following properties
		\begin{enumerate}
			\item[$i)$] $h_{F,S,A, \delta} \in \mathcal{C}^1(\er^3 \setminus S)$,
			\item[$ii)$] $h_{F,S,A, \delta}$ is injective,
			\item[$iii)$] $h_{F,S,A, \delta}(x) = x$ on $\er^3 \setminus E_{F,S,A}(\alpha_{A}, \delta)$,
			\item[$iv)$] $h_{F,S,A, \delta}(x) = x$ on $F$,
			\item[$v)$] $D_{\vec{n}}h_{F,S,A, \delta} = 0$ on $F\cap E_{F,S,A}(\alpha_{A}, \tfrac{\delta}{2}) \cap \{r<yt\} \supset F\cap B(A,\delta/2)$, where $\vec{n} \in  N_F$ is the normal vector to $F$,
			\item[$vi)$] $\|D h_{F,S,A, \delta}\|_{\infty}<C$ independent of choice of $\alpha_A$ and $\delta$.
		\end{enumerate} 
	\end{prop}
	
	\begin{proof}
		The proof of point $i)$ is a simple since the map is a combination of maps which are all smooth. The maps $\psi_{F,S,A}$ and $ \psi^{-1}_{F,S,A}$ are obviously smooth. The cylindrical coordinates are smooth away from $\er\times\{(0,0)\}$ and so is its inverse (note that $\psi_{F,S,A}(E_{F,S,A}(\alpha_{A}, \delta))$ stays far away from $-\pi$ and $\pi$ because we assume that angles at the vertices are small). Proposition~\ref{GString} and Proposition~\ref{CString} guarantee that the maps $g$ and $c$ are both smooth. We easily check the continuity of the partial derivatives at $\partial(E_{F,S,A}(\alpha_{A}, \delta)\cap \{x : 0<y<\delta, 0< r< 3yt\})\setminus S$ by $c_{2yt, 3yt}(3yt) = 1$ and $c_{\delta/2, \delta}(\delta) = 1$.
		
		Point $ii)$ is essentially a one-dimensional question. The map leaves $y$ and $r$ unchanged. Thus it suffices to show that for fixed $y$ and $r$ the mapping
		$$
		\varphi\mapsto (1-c_{\delta/2, \delta}(y))[(1-c_{2yt, 3yt}(r))g_{\alpha_A}(\phi) +  c_{2yt, 3yt}(r)\phi] + c_{\delta/2,\delta}(y)\phi 
		$$
		is injective. Indeed this function is a convex combination of $g_{\alpha_A}$ and $id$ which are both increasing on $[-\alpha_A, \alpha_A]$. Thus the function itself is increasing for each $y$ and $r$. Therefore we prove the injectivity of $h_{F,S,A, \delta}$.
		
		Point $iii)$ follows immediately from the definition of $h_{F,S,A,\delta}$ and the fact that $\tfrac{3}{2}\tan\tfrac{5\theta_{F,A}}{8} < \tan\tfrac{15\theta_{F,A}}{16}$. Point $iv)$ is also an immediate result of the definition, especially considering $F \subset \{\phi = 0\}$ and $g_{\alpha_A}(0) = 0$.
		
		Point $v)$ follows from the following considerations. On the set in question we have $y<\delta/2$ and $r<yt$ and therefore
		$$
		h_{F,S,A, \delta}(x) = \psi^{-1}_{F,S,A}\Big(\Phi\big(y, r, g_{\alpha_A}(\phi) \big) \Big).
		$$
		We calculate the derivative of $h_{F,S,A, \delta}$ using the chain rule. On the set $\psi_{F,S,A}(F)\subset \er^+_0\times\er^+_0\times\{0\}$ we have $\partial_y\Phi(y,r,0) = e_1$, $\partial_r\Phi(y,r,0) = e_2$, $\partial_{\phi}\Phi(y,r,0) = e_3$. Since the vector $\vec{n}$ is mapped by $\psi_{F,S,A}$ onto $\pm e_3$, and because $$
		\partial_y\Phi(y,r,0)\bot \psi_{F,S,A}(\vec{n}) = \pm e_3 \bot \partial_r\Phi(y,r,0)
		$$
		we have that
		$$
		\partial_{\vec{n}} h_{F,S,A, \delta}(x) = D\psi^{-1}_{F,S,A}\Big(\Phi\big(y, r, g_{\alpha_A}(\phi) \big) \Big)\partial_{\phi}\Phi\big(y, r, g_{\alpha_A}(\phi) \big).
		$$
		We have that $\partial_{\phi}\Phi\big(y, r, g_{\alpha_A}(\phi) \big) = g'_{\alpha_A}(0)r e_3$ but by Proposition~\ref{GString} we have that $g_{\alpha_A}'(0) = 0$, which yields (recall $|D\psi|=1$) claim $v)$.
		
		To prove point $vi)$ we calculate
		$$
		|\partial_{\phi}h_{F,S,A, \delta}(x)| = r \Big( (1-c_{\delta/2, \delta}(y))[(1-c_{2yt, 3yt}(r))g'_{\alpha_A}(\phi) +  c_{2yt, 3yt}(r)] + c_{\delta/2,\delta}(y) \Big) \leq C
		$$
		further
		$$
		|\partial_r h_{F,S,A, \delta}(x)| \leq \Big(1+ \tfrac{Cr |g_{\alpha_A}(\phi) - \phi|}{yt} \Big) \leq C
		$$
		and finally using Proposition~\ref{CString} we get
		$$
		|\partial_y h_{F,S,A, \delta}(x)| = \Big(1+ \tfrac{C|g_{\alpha_A}(\phi)-\phi|r}{\delta} + C\Big) \leq C
		$$
		because $r < \delta$.
	\end{proof}

	Recall $\theta_A =\min\{\tfrac{\pi}{4}, \theta_{S,S'}; A\subset S, S' \subset F\}$. 
	
	\begin{lemma}\label{AlbumI}
		For each side $S$ with endpoint $A$ and for each $0<\delta<\tfrac{d_A}{4}$ we define the set
		$$
		L_{S, A}(\delta) = \{x\in \er^3 :y\in(0,\delta), r \in (0,y \tan(\tfrac{1}{3}\theta_A)), \phi\in[-\pi,\pi) \}.
		$$
		Then for each $\tilde{A} \subset \tilde{S}$ 0-subsimplex and 1-subsimplex of $\T$ it holds that either
		$$
		(\tilde{A} = A \text{ and } \tilde{S} = S) \text{ or } L_{S,A}(\delta) \cap L_{\tilde{S} ,\tilde {A}}(\tilde{\delta}) = \emptyset.
		$$
		Also
		\begin{equation}\label{ConatinmentL}
		L_{S,A}(\delta) \subset S+B(0,\delta).
		\end{equation}
	\end{lemma}
	\begin{proof}
		On one hand, in the first case $A \neq \tilde{A}$ and then, because $L_{S,A}(\delta) \subset B(A,2\delta) \subset B(A, \tfrac{1}{2}d_A)$) and by the definition of $d_A$ we know that $B(A, \tfrac{1}{2}d_A) \cap B(\tilde{A}, \tfrac{1}{2}d_{\tilde{A}}) = \emptyset$. On the other hand if $A = \tilde{A}$ and $S \neq \tilde{S}$ then we note that the set $L_{S,A}(\delta)$ is contained in a cone originating at $A$ centred around $S$ under the angle $\tfrac{\theta_A}{3}$. By the definition of the angle $\theta_A$ these cones cannot intersect elsewhere than at $A$. On the other hand $A\notin L_{S,A}(\delta)$ and so the two sets are disjoint.
		
		The inclusion \eqref{ConatinmentL} follows from the facts that $L_{S,A}(\delta)$ is contained in a cone around $S$ with angle $\tfrac{\theta_A}{3}$ originating at $A$, further $ y\tan \tfrac{\theta_A}{3} < y\tan\theta_A < y$ and $y<\delta$.
	\end{proof}
	
	For any positive $\delta$ such that $4\delta<d_A$ we choose an arbitrary face $F$ containing $S$ and define
	$$
	H_{S,A,\delta}(x) =
	\psi^{-1}_{F,S,A}\Big(\Phi\big(y, (1-c_{\delta/2, \delta}(y))g_{y \tan( \tfrac{1}{3} \theta_A)}(r) + c_{\delta/2,\delta}(y) r, \phi \big) \Big)
	$$
	on the set $L_{S,A}(\delta)$ and
	$$
	H_{S,A,\delta}(x) = x
	$$
	for $x\in \er^3\setminus L_{S,A}(\delta)$.
	
	\begin{prop}\label{Sunset}
		The mapping $H_{S,A, \delta}$ has the following properties
		\begin{enumerate}
			\item[$i)$] $H_{S,A, \delta} \in \mathcal{C}^1(\er^3 \setminus A)$,
			\item[$ii)$] $H_{S,A, \delta}$ is injective,
			\item[$iii)$] $H_{S,A, \delta}(x) = x$ on $\er^3 \setminus L_{S,A}(\delta) $,
			\item[$iv)$] $H_{S,A, \delta}(x) = x$ on $S$,
			\item[$v)$] $D_{\vec{n}}H_{S,A, \delta} = 0$ on $S \cap \{0<y<\tfrac{\delta}{2}\}$, for any $\vec{n}$ perpendicular to $S$,
			\item[$vi)$] $\|D H_{S,A, \delta}\|_{\infty}<C$ independent of $\delta$.
		\end{enumerate} 
	\end{prop}
	\begin{proof}
		The proof is similar to the proof of Proposition~\ref{HorrorMovieI}. Point $i)$ is a simple case of calculation. On the set $L_{S,A}(\delta) \setminus S$ the continuous derivative with respect to $y$ exists thanks to Proposition~\ref{GString}, Proposition~\ref{CString} and basic calculus. Further, on $L_{S,A}(\delta) \setminus S$, we calculate that the $r$ derivative is continuous thanks to Proposition~\ref{GString} and $\partial_{\phi}$ is constant. Now it is enough to check that the partial derivitives are continuous at $\partial L_{S,A}(\delta) \setminus S$. Since $g'_{y \tan(\tfrac{1}{3} \theta_A)}(y \tan(\tfrac{1}{3} \theta_A)) = 1$, $c_{\delta/2,\delta}(\delta) = 1$ and $c'_{\delta/2,\delta}(\delta) = 0$ this is easily checked.
		
		The derivative of $H_{S,A, \delta}$ on $S$ is easy to calculate, we use the basis
		$$
		\{u_1,u_2,u_3\} = \{\psi^{-1}_{F,S,A}(e_1), \psi^{-1}_{F,S,A}(e_2), \psi^{-1}_{F,S,A}(e_3)\}.
		$$
		Then $\partial_{u_1}H_{S,A, \delta}(x) = u_1$ and $\partial_{u_2}H_{S,A, \delta}(x) =\partial_{u_3}H_{S,A, \delta}(x) =0$ for any point $x \in S \cap \{0<y<\tfrac{\delta}{2}\}$. In fact 
		\begin{equation}\label{Spandex}
		\text{for any vector }v\in\spn\{u_2,u_3\} \text{ it holds } \partial_v H_{S,A,\delta}(x) = 0, x \in S \cap \{0<y<\tfrac{\delta}{2}\}.
		\end{equation}
		Therefore we have that $DH_{S,A,\delta}$ is continuous up to and including $\partial L_{S,A}(\delta)\setminus \{A\}$ which is point $i)$. Already we have proved point $v$ in \eqref{Spandex}.
		
		The injectivity claim of $ii)$ is easily checked. Clearly for any pair $(y,r,\phi)$ and $(y',r',\phi')$ where either $y\neq y'$ or $\phi \neq \phi'$ then $H_{S,A,\delta}(x)\neq H_{S,A,\delta}(x')$. Therefore it suffices to check that for any fixed $0<y<\delta$ and any $\phi\in(-\pi,\pi]$ that 
		$$
		r\mapsto(1-c_{\delta/2, \delta}(y))g_{y \tan( \tfrac{1}{3} \theta_A)}(r) + c_{\delta/2,\delta}(y) r
		$$
		is increasing. Because $g$ and $\id$ are both increasing and the given function is their convex combination this is true.
		
		Points $iii)$ and $iv)$ are obvious from the definition of $H_{S,A,\delta}$. The calculations for $v)$ and $vi)$ are the similar as in Proposition~\ref{HorrorMovieI} 
	\end{proof}
	
	For each $A$ 0-subsimplex of $\T$ we define the map $G_{A,\delta}$, for any positive $\delta$ such that $3\delta<d_A$, as follows
	$$
	G_{A,\delta}(x) =\begin{cases} A + \frac{x-A}{|x-A|}g_{\delta}(|x-A|) \quad &x\neq A\\
	A \quad &x=A.
	\end{cases}
	$$
	\begin{prop}\label{Lollipop}
		The mapping $G_{A, \delta}$ has the following properties
		\begin{enumerate}
			\item[$i)$] $G_{A, \delta} \in \mathcal{C}^1(\er^3)$,
			\item[$ii)$] $G_{A, \delta}$ is injective,
			\item[$iii)$] $G_{A, \delta}(x) = x$ on $\er^3 \setminus B(A,\delta) $,
			\item[$iv)$] $G_{A, \delta}(V) = V$ for every $V$ subsimplex of  $\T$,
			\item[$v)$] $DG_{A, \delta}(A) = 0$,
			\item[$vi)$] $\|D G_{A, \delta}\|_{\infty}<C$ independent of $\delta$.
		\end{enumerate} 
	\end{prop}
	\begin{proof}
		Point $i)$ follows from $G_{A,\delta}$ being a radial map constructed using smooth $g_{\delta}$ and point $ii)$ comes from $g$ being increasing. Point $iii)$ is because $g_{\delta}(r) = r$ for $r>\delta$. Point $iv)$ is thanks to small choice of $\delta$ so that $B(A,\delta)$ intersects only those simplices that intersect $A$. Point $v)$ is because $g_{\delta}'(0) = 0$. Point $vi)$ follows directly from Proposition~\ref{GString}.
	\end{proof}
	
	\begin{lemma}\label{Lemmings}
		Let $S$ be a 1-subsimplex of $\T$ contained in $F$ a 2-subsimplex of $\T$. We denote the endpoints of $S$ by $A$ and $A'$ vertices of $\T$ and $\ell_S = |A-A'| = \H^1(S)$. For each $\delta < \tfrac{d_S}{4}$ and each $\omega$ we define the the sets
		$$
		\begin{aligned}
		D_{F,S}(\delta) = \Big\{x\in& \er^3; x = \psi^{-1}_{F,S,A}(\Phi(y,r,\phi));\\
		& y\in[\tfrac{\delta}{4}, \ell_S - \tfrac{\delta}{4}], r\in[0,\tfrac{1}{4}\delta\tan(\tfrac{1}{3}\theta_{S})], \phi\in[-\tfrac{\omega_S}{3}, \tfrac{\omega_S}{3}]\Big\}.
		\end{aligned}
		$$
		Then for any $S,\tilde{S}$ 1-subsimplices of $\T$ and any $F, \tilde{F}$ 2-subsimplices of $\T$, $S\subset F$ and $\tilde{S} \subset \tilde{F}$ exactly one of the following holds
		\begin{enumerate}
			\item $F = \tilde{F}$ and $S = \tilde{S}$
			\item $F \neq \tilde{F}$ and $S = \tilde{S}$ implying 
			$$
			D_{F,S}(\delta) \cap D_{\tilde{F},\tilde{S}}(\tilde{\delta}) \subset S
			$$
			\item $S \neq \tilde{S}$ and
			$$
			D_{F,S}(\delta) \cap D_{\tilde{F},\tilde{S}}(\tilde{\delta}) = \emptyset.
			$$
		\end{enumerate}
		Further
		\begin{equation}\label{ConatinmentD}
		D_{F,S}(\delta) \subset S+B(0,\tfrac{1}{4}\delta).
		\end{equation}
	\end{lemma}
	
	\begin{proof}
		The inclusion \eqref{ConatinmentD} is obvious since $\tan(\tfrac{1}{3}\theta_{S}) < \tan(\theta_{S}) \leq 1$. From here, by the definition of $d_S$ and $2\delta < d_S$ we easily see, when $S \cap \tilde{S} = \emptyset$ then
		$$
		D_{F,S}(\delta) \cap D_{\tilde{F},\tilde{S}}(\tilde{\delta}) = \emptyset.
		$$
		
		On the other hand if $\{A\} = S \cap \tilde{S}$ then we reason as follows. It is not hard to calculate that
		$$
			\bigcup_{x\in S\setminus A} B(x,\tfrac{1}{4}|x-A|\tan(\tfrac{1}{3}\theta_S)) \Subset K_{A, S, \theta_S/2}
		$$
		where
		$$
			K_{A, S, \theta_S/2} = \psi^{-1}_{F,S,A}(\Phi(y,r,\phi)); y>0, r\in[0,y\tan(\tfrac{1}{2}\theta_{S})), \phi\in(-\pi, \pi]\Big\}
		$$
		is the cone with vertex at $A$ along $S$ and with angle to its axis of $\theta_S/2$. By the definition of $\theta_S$ these cones are disjoint. Therefore we have $D_{F,S}(\delta) \cap D_{\tilde{F},\tilde{S}}(\tilde{\delta}) = \emptyset$.

		Finally we consider the case that $S = \tilde{S}$. Each point of $D_{F,S} \setminus S$ can be described as the rotation around $S$ of a point in $D_{F,S} \cap F$ with the rotation of at most $\tfrac{1}{3}\omega_S$ on either side of $F$. The same holds for $D_{F',S}$. Since the angle between the faces $F$ and $F'$ at $S$ is at least $\omega_S$. Therefore the rotation of $\tfrac{1}{3}\omega_S + \tfrac{1}{3}\omega_S<\omega_S$ is so small that the sets $D_{F,S}$ and $D_{F',S}$ can only meet at $S$.
	\end{proof}
	
	We define the mapping $b_{F,S, \delta}$ for any side $S$ and any $0<\delta<\tfrac{d_S}{4}$ as follows. Call $t = \tan(\tfrac{1}{3}\theta_{S})$. On the set $(\er^3\setminus D_{F,S}(\delta)) \cup S$ we define
	$$
	b_{F,S, \delta}(x) = x.
	$$
	To simplify notation call $c_1(y) = c_{\delta/4, \delta/2}(y) - c_{\ell_S-\delta/2, \ell_S-\delta/4}(y)$ and $c_2(r) =c_{\delta t/8,\delta t/4}(r) $. On the set $D_{F,S}(\delta)$ we define
	$$
	b_{F,S, \delta}(x) =\psi_{F,S,A}^{-1}\Big(\Phi\big(y, r, c_1(y)[(1-c_2(r))g_{\omega_{S}}(\phi) + c_2(r)\phi]  + (1 - c_1(y))\phi\big) \Big).
	$$
	
	\begin{prop}\label{Aaaggrr}
		The mapping $b_{F,S, \delta}$ has the following properties
		\begin{enumerate}
			\item[$i)$] $b_{F,S, \delta} \in \mathcal{C}^1(\er^3 \setminus S)$,
			\item[$ii)$] $b_{F,S, \delta}$ is injective,
			\item[$iii)$] $b_{F,S, \delta}(x) = x$ on $\er^3 \setminus D_{F,S}(\delta)$,
			\item[$iv)$] $b_{F,S, \delta}(x) = x$ on $F\cap D_{F,S}(\delta) \supset S$,
			\item[$v)$] $D_{\vec{n}}b_{F,S, \delta} = 0$ on $F \cap \{x : y\in[\tfrac{\delta}{2}, \ell_S - \tfrac{\delta}{2}] , 0\leq r \leq \tfrac{1}{8}\delta t \}$, where $\vec{n}$ is the normal vector to $F$,
			\item[$vi)$] $\|D b_{F,S, \delta}\|_{\infty}<C$ independent of $\delta$.
		\end{enumerate} 
	\end{prop}
	
	\begin{proof}
		The map $b_{F,S, \delta}$ is the same as the map $h_{F,S,A, \delta}$ with the only difference that $b_{F,S, \delta}$ is defined along an entire side $S$ away from its corners but $h_{F,S,A,\delta}$ is only defined near the corners. The proof therefore is basically a repetition of the arguments from Proposition~\ref{HorrorMovieI}.
	\end{proof}

	\begin{lemma}\label{LemmingsAgain}
		For every 1-subsimplex of $\T$, $S$, with endpoints $A$ and $A'$. For every $0<\delta < d_S$ we define the set $W_{S}(\delta)$ as
		$$
		\begin{aligned}
		W_{S}(\delta) = \Big\{x\in& \er^3; x = \psi^{-1}_{S}(\Phi(y,r,\phi));\\
		& y\in[\tfrac{\delta}{4}, \ell_S - \tfrac{\delta}{4}], r\in[0,\tfrac{1}{4}\delta\tan(\tfrac{1}{3}\theta_S)], \phi\in[-\pi,\pi)\Big\}.
		\end{aligned}
		$$
		For each $0<\delta, \delta'<d_S$ we have that either $S = S'$ or $W_{S}(\delta) \cap W_{S'}(\delta') = \emptyset$. Also
		\begin{equation}\label{ConatinmentW}
		W_S(\delta) \subset S+B(0,\tfrac{1}{4}\delta).
		\end{equation}
	\end{lemma}
	
	\begin{proof}
		Let us consider the case where $S\cap S' = \emptyset$. The inclusions
		$$
		W_{S}(\delta)
		\subset S+B(0,\tfrac{\delta\tan\tfrac{1}{3}\theta_S}{4})
		\subset S+B(0,\tfrac{\delta}{4})
		$$
		hold by $\tan\theta_S < 1$. This proves \eqref{ConatinmentW}. Further, for $S\cap S' = \emptyset$, by the fact that $\delta< d_S \leq \dist(S,\T_S) \leq \dist(S,S')$ we have that $S+B(0,\tfrac{d_S}{4}) \cap S'+B(0,\tfrac{d_{S'}}{4}) = \emptyset$ and so $W_{S}(\delta)\cap W_{S'}(\delta) = \emptyset$.
		
		Thus the case remains that $S\cap S' = A$ a common endpoint of $S$ and $S'$. As in Lemma~\ref{Lemmings} it holds that
		$$
		W_{S}(\delta) \subset K_{A, S, \theta_S/2} \text{ and } W_{S'}(\delta) \subset K_{A, S', \theta_S/2}
		$$
		and $K_{A, S, \theta_S/2}\cap K_{A, S', \theta_S/2} = \emptyset$ for $S\neq S'$.
	\end{proof}
	
	For each side $\{A,A'\}\subset S\subset F$ (let $F$ be any one of the faces containing $S$ chosen arbitrarily) we define the mapping $w_{S, \delta}$ as follows. We use the notation $t = \tan(\tfrac{1}{3}\theta_S)$ and $c(y) = c_{\delta/4, \delta/2}(y) - c_{\ell_S-\delta/2, \ell_S-\delta/4}(y)$. On $W_S(\delta)$ we define
	$$
	w_{S, \delta}(x) =\psi_{F,S,A}^{-1}\Big(\Phi\big(y, c(y)g_{\delta t/4}(r) + (1-c(y))r, \phi \big) \Big)
	$$
	and on $W_S(\delta)$ we put
	$$
	w_{S, \delta}(x)  = x.
	$$
	
	\begin{prop}\label{Aaaaaaaggggrrr}
		The mapping $w_{S, \delta}$ has the following properties
		\begin{enumerate}
			\item[$i)$] $w_{S, \delta} \in \mathcal{C}^1(\er^3 )$,
			\item[$ii)$] $w_{S, \delta}$ is injective,
			\item[$iii)$] $w_{S, \delta}(x) = x$ on $\er^3 \setminus W_{S}(\delta)$,
			\item[$iv)$] $w_{S, \delta}(x) = x$ on $S$, and $w_{S, \delta}(V) = V$ for every $V$ subsimplex of  $\T$
			\item[$v)$] $D_{\vec{n}}w_{S, \delta} = 0$ on $S \cap \{x : y\in[\tfrac{\delta}{2}, \ell_S - \tfrac{\delta}{2}]\}$, where $\vec{n}$ is any vector perpendicular to $S$,
			\item[$vi)$] $\|D w_{S, \delta}\|_{\infty}<C$ independent of $\delta$.
		\end{enumerate} 
	\end{prop}
	
	\begin{proof}
		The map $w_{S, \delta}$ is the same as the map $H_{S,A, \delta}$ with the only difference that $H_{S, A, \delta}$ is defined along an entire side $S$ away from its corners but $H_{S,A,\delta}$ is only defined near the corners. The proof therefore is basically a repetition of the arguments from Proposition~\ref{HorrorMovieI}.
	\end{proof}

	For every $F$ 2-subsimplex of $\T$ we choose a 1-subsimplex of $F$ called $S$ arbitrarily and choose one of its endpoints called $A$ also arbitrarily. We use the shorthand
	$$
	\begin{aligned}
	p_1(x) &= p_{F,S,A}^1(x) = [\psi_{F,S,A}(x)]^1,\\
	p_2(x) &= p_{F,S,A}^2(x) = [\psi_{F,S,A}(x)]^2, \\
	p_3(x) &= p_{F,S,A}^3(x) = [\psi_{F,S,A}(x)]^3.
	\end{aligned}
	$$
	
	\begin{lemma}\label{NoMoreLemmings}
		Let $0<2\delta<d_F$ and call $t = \tan\tfrac{1}{3}\omega_F$. For every 2-subsimplex of $\T$, $F$, with sides $S_1,S_2,S_3$ we define the set
		$$
		K_F(\delta ) = \{x\in \er^3 : \lambda_{F,\delta}(x)>0; |p_3|  < \tfrac{1}{2} \delta t \lambda_{F,\delta}(x) \}
		$$
		Then either $F = F'$ or for any $\delta' < d_{F'}$ it holds that $K_F(\delta) \cap K_{F'}(\delta') = \emptyset$.
		Also
		\begin{equation}\label{ConatinmentK}
		K_F(\delta) \subset F+B(0,\delta).
		\end{equation}
	\end{lemma}
	\begin{proof}
		The proof is similar to the lemmata above. Take $F\neq F'$, either $F\cap F' = S$, $F\cap F' = \{A\}$ or $F\cap F'  =\emptyset$.
		
		In the $F\cap F'  =\emptyset$ case, by the choice of $\omega_F \leq \tfrac{\pi}{4}$ and because $\delta < \tfrac{1}{2}d_F$ we have $K_F(\delta) \subset F+B(0,\delta) \subset F+B(0,\tfrac{1}{2}d_F)$. But by the definition of $d_F$ it holds that $F+B(0,\tfrac{1}{2}d_F) \cap F'+B(0,\tfrac{1}{2}d_{F'}) = \emptyset$.
		
		In the case where $F$ and $F'$ share exaclty one side $S$ then the choice of $\omega_{F}$ garantees that any point $(p_1,p_2,p_3) \in K_F(\delta)$ satisfies
		$$
		\angle (p_1,p_2,p_3), (p_1,0,0), (p_1,p_2,0) \leq \tfrac{1}{3}\omega_F
		$$
		for any $(p_1,0,0) \in S \subset F$. Since the same holds at other faces then, by the definition of $\omega_F$ we see that $K_F(\delta) \cap K_{F'}(\delta') = \emptyset$.
		
		By the previous estimate we know that $K_F(\delta) \subset T_1, T_2$ where $T_1,T_2 \in \T$ are the two uniquely determined simplices both containing $F$. Assume that $F,F'$ are a pair of 2-subsimplices of $\T$ such that $F\cap F' = \{A\}$. Then it must hold that $T_1,T_2, T_1', T_2'$ the simplices containing $F$ and $F'$ respectively are all distinct simplices. Therefore obviously $(T_1 \cup T_2)^{\circ} \cap (T_1'\cup T_2')^{\circ} = \emptyset$. But since $K_{F}(\delta) \subset (T_1 \cup T_2)^{\circ}$ then necessarilly $K_F(\delta) \cap K_{F'}(\delta') \emptyset$.
	\end{proof}

	We denote $t = \tan \tfrac{1}{3}\omega_F$. Notice that this $t$ is less than or equal to $t$ from Proposition~\ref{Aaaggrr} by definition. For every $\delta > 0$  and for every $F$ there exists a smooth function $\lambda_{F,\delta}$ such that
	$$
	\lambda_{F, \delta} (x) = 
	\begin{cases}
	0 \quad &\dist\big( (p_1(x), p_2(x), 0), \er^2\times\{0\} \setminus \psi_{F,S,A} (F)\big) <\frac{t\delta}{32}\\
	1 \quad &\dist\big( (p_1(x), p_2(x), 0), \er^2\times\{0\} \setminus \psi_{F,S,A} (F)\big) >\frac{t\delta}{16},
	\end{cases}
	$$
	$\lambda_{F,\delta}(x)$ is independent on the value of $p_3$ and $\lambda_{F,\delta}$ increases as $\dist\big( (p_1(x), p_2(x), 0), \er^2\times\{0\} \setminus \psi_{F,S,A} (F)\big)$ increases.

	For each $0< \delta< d_{F}$ we define the map
	$$
	s_{F,\delta}(x) = \psi_{F,S,A}^{-1}\big(p_1(x),p_2(x),\lambda_{F,\delta}(x) g_{\tfrac{1}{64} \delta t^2}(p_3(x)) + (1-\lambda_{F,\delta}(x))p_3(x)\big).
	$$
	
	\begin{prop}\label{BahHumbug}
		The mapping $s_{F, \delta}$ has the following properties
		\begin{enumerate}
			\item[$i)$] $s_{F, \delta} \in \mathcal{C}^1(\er^3, \er^3)$,
			\item[$ii)$] $s_{F, \delta}$ is injective,
			\item[$iii)$] $s_{F, \delta}(x) = x$ on $\{|p_3(x)| > \tfrac{1}{64} \delta t^2 \} \cup \{ \lambda  = 0 \}$,
			\item[$iv)$] $s_{F, \delta}(x) = x$ on every $V$ subsimplex of  $\T$,
			\item[$v)$] $D_{\vec{n}}s_{F, \delta} = 0$ on $F\cap \{\lambda_{F,\delta} = 1 \}$, where $\vec{n}$ the normal vector of $F$,
			\item[$vi)$] $|Ds_{F, \delta}| < C$ independent of $\delta$.
		\end{enumerate} 
	\end{prop}
	\begin{proof}
		We calculate
		$$
		D s_{F, \delta} = D\psi_{F,S,A}^{-1} \cdot D\big(p_1(x),p_2(x),\lambda_{F,\delta}(x) g_{\tfrac{1}{64} \delta t^2}(p_3(x)) + (1-\lambda_{F,\delta}(x))p_3(x)\big).
		$$
		The term $D\psi_{F,S,A}^{-1}$ is a constant $D\lambda_{F,\delta}$ exists and is continuous and the same holds for $g$. The functions $p_1, p_2, p_3$ are simply coordinate functions of identity in a different basis. Since all of the above derivatives exist and are continuous so is $Ds_{F,\delta}$, which is point $i)$.
		
		The injectivity of the map $s_{F,\delta}$ is a 1-dimensional question. The map $s_{F,\delta}$ is constant with respect to the projection in $p_1,p_2$ directions. For each $(p_1,p_2)$ the value of $\lambda_{F,\delta}$ is constant with respect to $p_3$. Therefore each segment $(p_1,p_2) \times[-\tfrac{1}{64} \delta t^2,\tfrac{1}{64} \delta t^2]$ is mapped onto itself. Moreover for every value of $\lambda_{F,\delta} \in [0,1]$ we have that both $g(p_3)$ and $p_3$ are increasing in $p_3$. Therefore also $\lambda_{F,\delta}g_{\tfrac{1}{64} \delta t^2}(p_3) + (1-\lambda_{F,\delta})p_3$ is increasing in $p_3$. Thus $s_{F,\delta}$ is increasing on each segment $(p_1,p_2) \times[-\tfrac{1}{64} \delta t^2,\tfrac{1}{64} \delta t^2]$. As each segment is mapped onto itself we see that the map $s_{F,\delta}$ must be injective, which is point $ii)$.
		
		Since $g_{\tfrac{1}{64} \delta t^2}(p_3) = p_3$ for all $|p_3|>\tfrac{1}{64} \delta t^2$ we immediately get point $iii)$ from the definition of $s_{F,\delta}$.  Point $iv)$ holds for $F$ because $g_{\tfrac{1}{64} \delta t^2}(0) = 0$. On other $V$ subsimplices of $\T$, point $iv)$ follows from point $iii)$ and Lemma~\ref{NoMoreLemmings}. Point $v)$ follows from $g'_{\tfrac{1}{64} \delta t^2}(0) = 0$. Point $vi)$ follows from two facts. The first of these is Proposition~\ref{GString}, that $|g'_{\delta}| \leq C$. The second fact is $|D\lambda_{F,\delta}| \approx \delta^{-1}t^{-2}$ and $|g_{\delta t^2}(\cdot) - \id(\cdot)| \leq \delta t^2$, which implies that $D|\lambda_{F,\delta}(\cdot)[g_{\tfrac{1}{64}\delta t^2}(\cdot) - \id(\cdot)]| \leq C$.
	\end{proof}
	
	\begin{proof}[Proof of Theorem~\ref{SocialConstruct}]
		For each $T_i$ we have a given $\delta_i$. Without loss of generality we assume that $\delta_i < \min\{d_F; F\cap T_i\neq \emptyset \}$. Let $A$ be a vertex of $\T$ we define $\delta_A = \min\{\delta_i; A\in T_i\}$ further we define $\delta_S = \min\{\delta_A; A\in S\}$ and $\delta_F = \min\{\delta_S; S \subset F\}$.
		
		For each pair $A$ and $F$, $A\in F$ with the two sides $S,S'$ such that $A\in S,S' \subset F$ we define the map $h_{F,A} = h_{F,S,A,\delta_A} \circ h_{F,S',A,\delta_A}$.The order of the maps in the composition is immaterial for the validity of the theorem. We define the map $h$ as the composition of all $h_{F,A}$ for all pairs $F$ and $A$. By \eqref{Contained} \eqref{Disjoint} and by Proposition~\ref{HorrorMovieI} $iii)$ and $iv)$ we have that for all $(F,A) \neq (F',A')$ we have $\{h_{F,A} \neq \id\}\cap \{h_{F',A'} \neq \id\} = \emptyset$ and hence the order of the composition of the maps $h_{F,A}$ in the construction of $h$ is immaterial. Similarly we define $H$ as the composition of $H_{S,A, \delta_A}$ for all pairs $(S,A)$, $A$ endpoint of $S$. Again by Lemma~\ref{AlbumI} and Proposition~\ref{Sunset} $iii)$ for all $(S,A) \neq (S',A')$ we have $\{H_{S,A, \delta_A} \neq \id\}\cap \{H_{S',A', \delta_{A'}} \neq \id\} = \emptyset$ and hence the order of the composition of the maps $H_{S,A, \delta_A}$ in the construction of $H$ is immaterial. Also we define $G$ as the composition of $G_{A, \delta_A}$ for all $A$. Again by $\delta_A < \tfrac{d_A}{3}$, for all $A \neq A'$ we have $\{G_{A, \delta_A} \neq \id\}\cap \{G_{A', \delta_{A'}} \neq \id\} = \emptyset$ and hence the order of the composition of the maps $G_{A, \delta_A}$ in the construction of $G$ is immaterial.
		
		We define the map $b$ as the composition of all $b_{F,S, \delta_S}$ for all pairs $F$ and $S$. By Lemma~\ref{Lemmings} and Proposition~\ref{Aaaggrr} $iii)$ and $iv)$, for all $(F,S) \neq (F',S')$ we have $\{b_{F,S, \delta_S} \neq \id\}\cap \{b_{F',S', \delta_{S'}} \neq \id\} = \emptyset$ and hence the order of the composition of the maps $b_{F,S, \delta_{S}}$ in the construction of $b$ is immaterial. Similarly we define the map $w$ as the composition of all $w_{S, \delta_S}$ for all $S$. By Lemma~\ref{LemmingsAgain} and Proposition~\ref{Aaaaaaaggggrrr} $iii)$ we have that for all $S \neq S'$ we have $\{w_{S, \delta_S} \neq \id\}\cap \{w_{S', \delta_{S'}} \neq \id\} = \emptyset$ and hence the order of the composition of the maps $w_{S, \delta_{S}}$ in the construction of $w$ is immaterial. Finally we define $s$ as the composition of all $s_{F, \delta_F \tan(\tfrac{1}{3}\omega_F)}$ for all $F$. By Lemma~\ref{NoMoreLemmings} and Proposition~\ref{BahHumbug} $iii)$ we have that for all $F \neq F'$ we have $\{s_{F, \delta_F\tan(\tfrac{1}{3}\omega_F)} \neq \id\}\cap \{s_{F', \delta_{F'}\tan(\tfrac{1}{3}\omega_{F'})} \neq \id\} = \emptyset$ the order of the composition of the maps $s_{F, \delta_{F}\tan(\tfrac{1}{3}\omega_F)}$ in the construction of $s$ is immaterial.
		
		Let $x \in F$ be a point in any face, a 2-subsimplex, but $x$ is not contained in any side $S$ a 1-subsimplex. Then the normal space to $F$ at $x$, $N_F$ is the 1 dimensional space generated by $\vec{n}$. If $x \in F\cap \{\lambda_{F,\delta} = 1 \}$ then by Proposition~\ref{BahHumbug} point $v)$ we have $D_{\vec{n}}s_{F, \delta} = 0$. The set $F\cap \{\lambda_{F,\delta} = 1 \}$ is all of $F$, with distance greater than $\tfrac{1}{16}\delta \tan\tfrac{\omega_F}{3}$ from a 1-subsimplex $S$ of $F$. Supposing $\dist(x, S)< \tfrac{1}{16}\delta \tan\tfrac{\omega_F}{3} < \tfrac{1}{8}\delta\tan\tfrac{\theta_S}{3}$ and $[\Psi_{F,S,A}(x)]^2 \in [\delta/2, \ell_S-\delta/2]$ then by Proposition~\ref{Aaaggrr} point $v)$ we have $D_{\vec{n}}b_{F,S, \delta_S}(x) = 0$. Finally in the case that $x\in B(A,\delta/2)\cap F$ then by Proposition~\ref{HorrorMovieI} point $v)$ we have that $D_{\vec{n}}h_{F,A}(x) = 0$.
		
		Let $x\in S$ then either $x \in S \cap \{ y\in[\tfrac{\delta_S}{2}, \ell_S - \tfrac{\delta_S}{2}]\}$ and by Proposition~\ref{Aaaaaaaggggrrr} point $v)$ it holds that $D_{\vec{n}}w_{S, \delta_S} = 0$ for any $\vec{n}$ perpendicular to $S$. In the other case there is an $A$, a 0-subsimplex of $S$, such that $|x-A| < \delta_S/2$. But then by Proposition~\ref{Sunset} point $v)$ we have $D_{\vec{n}}H_{S,A, \delta_A} = 0$ on $S \cap \{0<y<\tfrac{\delta}{2}\}$, for any $\vec{n}$ perpendicular to $S$. In the case that $x$ is a 0-subsimplex, then by Proposition~\ref{Lollipop} point $v)$, we have $DG_{A, \delta_A}(A) = 0$.
		
		
		We define
		$$
		\Xi_{\delta_i} =s\circ w \circ b\circ G \circ H\circ h .
		$$
		The map $\Xi_{\delta_i}$ is $\mathcal{C}^1(\er^3, \er^3)$ because each map $h,H,G,b,w,s$ is smooth except for some parts of the boundaries of simplices of $\T$. Whenever the partial derivatives of one of these maps doesn't exist then it is in a normal direction to a subsimplex of $\T$. By the previous two paragraphs however we have that at those points that map is composed with another map that has normal derivitive at those points equal to zero. Therefore we have that all normal derivitives at all points on $\partial T$ for any $T\in \T$ is zero. Further each of the maps $h,H,G,b,w,s$ is injective and has the property that it sends $V$ onto $V$ for all $V$ subsimplices of $\T$. Therefore $\Xi_{\delta_i}(V)  = V$ for all $V$ subsimplices of $\T$ (including the case that $V = T_i$). Thus we have $i), ii)$, $iii)$ and $iv)$.
		
		The property $v)$ comes from Proposition~\ref{HorrorMovieI} $vi)$, Proposition~\ref{Sunset} $vi)$, Proposition~\ref{Lollipop} $vi)$, Proposition~\ref{Aaaggrr} $vi)$, Proposition~\ref{Aaaaaaaggggrrr} $vi)$ and Proposition~\ref{BahHumbug} $vi)$.
		
		Point $vi)$ is a result of \eqref{ConatinmentE} with Proposition~\ref{HorrorMovieI} $iii)$, \eqref{ConatinmentL} with Proposition~\ref{Sunset} $iii)$, Proposition~\ref{Lollipop} $iii)$, \eqref{ConatinmentD} with Proposition~\ref{Aaaggrr} $iii)$, \eqref{ConatinmentW} with Proposition~\ref{Aaaaaaaggggrrr} $iii)$ and \eqref{ConatinmentK} with Proposition~\ref{BahHumbug} $iii)$.
	\end{proof}
	
	\begin{proof}[Proof of Theorem~\ref{main} for $n=3$]
		For $f$ we have $\T = \{T_i; i\in \en\}$ a locally finite simplicial complex such that $f$ is affine on each $T_i$. Let $\epsilon>0$ be a fixed number. Let $p_i>0$ be chosen such that
		$$
		\phi_X(T_i \cap [\partial T_i + B(0,p_i)]) < \tfrac{2^{-i}\epsilon}{2(1+\|D\Xi\|_{\infty})|Df_{|T_i}|}.
		$$
		Then we put
		$$
		\delta_i = \min\{d_F, F\subset T_i\}\cup\{p_i\}.
		$$
		
		We define $\tilde{f} = f\circ\Xi_{\delta_i}$. Then by Theorem~\ref{SocialConstruct} $iv)$, $ii)$ and $i)$ we have that $\tilde{f}\in \mathcal{C}^1(\Omega, \er^3)$ is a homeomorphism. On the other hand we calculate using Theorem~\ref{SocialConstruct} $v)$ and $vi)$ and the definition of $\delta_i$ that
		$$
		\begin{aligned}
		\|(Df - D\tilde{f})\chi_{T_i}\| 
		&\leq \|Df \chi_{T_i \cap \{ f\neq\tilde{f}\}}\| + \|D\tilde{f}\chi_{T_i \cap \{ f\neq\tilde{f}\}}\|\\
		& \leq \|Df_{|T_i}\|_{\infty}\phi_{X}(T_i \cap [\partial T_i + B(0,\delta_i)])(1 + [1+\|D\Xi_{\delta_i}\|_{\infty}])\\
		&\leq 2^{-i} \epsilon.
		\end{aligned}
		$$
		Then $\|Df - D\tilde{f}\|_X \leq \sum_i \|(Df - D\tilde{f})\chi_{T_i}\|_X \leq \epsilon$.
	\end{proof}
	
	\begin{remark}
		If we repeat the entire above construction with $g^{\alpha}_{\delta}$ instead of $g_{\delta}$, where
		$$
		g_{\delta}^{\alpha}(t) = \frac{1+\alpha}{\delta^{\alpha}} t^{1+\alpha}  - \frac{\alpha}{\delta^{1+\alpha}} t^{2+\alpha}
		$$
		then we can also achieve $\|\tilde{f}^{-1} - f^{-1}\|_{W^{1,p}} < \epsilon$. This is because Proposition~\ref{GString} holds also for $g_{\delta}^{\alpha}$ and for given $p$ it is not hard to  find an $\alpha > 0$ such that $\tilde{f}^{-1} \in W^{1,p}$ and as $\delta \to 0$ the norms of the difference of the two maps tends to 0.
	\end{remark}
	
	\section{Remark on the proof for higher dimension}
	The proof of Theorem~\ref{main} for $n\geq 4$ is in principle the same as for  dimension 2 and 3. The aim is to construct a mapping corresponding to $\Xi_{\delta}$. For this we do the following. Given a $k$-subsimplex $S = S_k$, $0\leq k \leq n-1$ we assume an orthonormal basis $v_1,\dots v_n$ where $v_1, \dots, v_k$ generate the $k$-dimensional surface of $S_k$ and $v_{k+1}, \dots, v_{n}$ are perpendicular to it. We call $x = (\tilde{x}, \hat{x}) \in \rn$, where $\tilde{x} \in \spn\{v_1,\dots, v_k\}$ and $\hat{x} \in \spn\{v_{k+1},\dots, v_n\}$ and $\tilde{p}(x) = \tilde{x}$ and $\hat{p} (x) = \hat{x}$ the corresponding orthogonal projections. For all $\delta >0$ we assume a smooth map $d_{S,\delta}$ such that
	\begin{equation}\label{Similar}
	d_{S,\delta}(\tilde{x}) =\begin{cases}
	0 \qquad &\tilde{x} \in \tilde{p}(\rn) \setminus \tilde{p}(S_k)\\
	0 \qquad &\tilde{x} \in \tilde{p}(S_k) \text{ and } \dist(\tilde{x} , \partial_k\tilde{p}(S_k) \leq \delta)\\
	1 \qquad &\tilde{x} \in \tilde{p}(S_k) \text{ and } \dist(\tilde{x} , \partial_k\tilde{p}(S_k) \geq 2\delta)\\
	\end{cases}	
	\end{equation}
	(where $\partial_k$ is the topological boundary in $\er^{k}$) with $|D d_{\delta, S}| \leq C \delta^{-1}$. As before we choose a small number $\alpha >0$ depending on the geometry of the simplicial complex near $S$. Then we define a map
	$$
	b_{S, \delta}(\tilde{x}, \hat{x}) = d_{S,\alpha^{k-1}\delta}(\tilde{x}) g_{\alpha^k\delta}(|\hat{x}|)\frac{\hat{x}}{|\hat{x}|} + (1-d_{S,\delta}(\tilde{x}))\hat{x} + \tilde{x}.
	$$
	By choosing $\alpha = \alpha_S$ small we get that any such pair of distinct maps $b_{S, \delta}$ and $b_{T, \delta}$ \emph{act independently} i.e. 
	$$
	\{b_{S, \delta} \neq \id\} \cap \{b_{T, \delta}\neq \id\} = \emptyset.
	$$
	The map $b_{S,\delta}$ pinches the central part of $S_k$ as a $n-k$ dimensional radial map perpendicularly to $S_k$ and behaves like identity in the other directions
	
	Let $S'$ be a $k-1$-subsimplex of $S$ and assume that (true up to rotate the basis and a translation) $S' = S \cap \{x\in \rn, x \cdot v_k = 0 \}$ and $x\cdot v_k \geq 0$ for $x\in S$. Then $\tilde{x} = (\tilde{\tilde{x}}, \tilde{x}_k)$ and if $\tilde{x}\in S'$ then $\tilde{x} = (\tilde{\tilde{x}} ,0)$. Let $d_{S', \delta}$ be the smooth map equaling 1 inside $S'$ and $2\delta$ from its ($k-1$ dimensional) boundary and equaling 0 at points $\delta$ close to its ($k-1$ dimensional) boundary and outside $S'$. This is similar to $d_{S,\delta}$ from \eqref{Similar}.  We define
	$$
	\begin{aligned}
	b_{S,S',\delta}(x) =
	& d_{S',\alpha^{k-1}\delta}(\tilde{\tilde{x}})\Big[c_{2\delta\alpha^{k-1}, 3\delta \alpha^{k-1}}(x_k) \hat{x} + (1-c_{2\delta\alpha^{k-1}, 3\delta \alpha^{k-1}}(x_k))g_{\alpha^k \delta x_k}(|\hat{x}|)\frac{\hat{x}}{|\hat{x}|} \Big]\\
	& + (1-d_{S',\alpha^{k-1}\delta}(\tilde{\tilde{x}}))\hat{x}+ \tilde{x}.
	\end{aligned}
	$$
	Again, assuming that the number $\alpha>0$ is small enough (chosen with respect to the geometry of the simplicial complex around $S$), we have that any pair  $S' \neq S''$ (both $k-1$-subsimplices of $S$) then $b_{S,S',\delta}$ and $b_{S,S'',\delta}$ act independently.
	
	We continue similarly for all $j$-subsimplices of $S$ for $j=0,1, \dots, k-2$. Then we compose all of the maps we have constructed. This gives us a $a_{S,\delta}$ with $D_{\vec{n}}a_{S,\delta} = 0$ on $S$ for any $\vec{n} \in \spn\{v_{k+1}, \dots , v_n\}$. The composition of all $a_{S,\delta}$ for all $j$-subsimplices $j=0,1, \dots, n-1$ gives a $\Xi_{\delta}$.
	

\end{document}